\documentclass[a4, 14pt]{amsart}
\usepackage{amssymb}
\usepackage{amstext}
\usepackage{amsmath}
\usepackage{amscd}
\usepackage{latexsym}
\usepackage{amsfonts}
\usepackage{color}
\usepackage{enumerate}
\usepackage{textcomp}
\usepackage[all]{xy}
\usepackage{graphicx}
\usepackage{appendix}
\usepackage[]{algorithm2e}

\usepackage{hyperref}
\hypersetup{
   colorlinks=true, 
    linkcolor=blue, 
    urlcolor=green, 
}

\usepackage{amsrefs}

\DeclareFontFamily{OT1}{rsfs}{}
\DeclareFontShape{OT1}{rsfs}{n}{it}{<-> rsfs10}{}
\DeclareMathAlphabet{\mathscr}{OT1}{rsfs}{n}{it}

\theoremstyle{plain}
\newtheorem{thm}{Theorem}[section]
\newtheorem*{thm*}{Theorem}
\newtheorem*{cor*}{Corollary}
\newtheorem*{defn*}{Definition}

\newtheorem{prop}[thm]{Proposition}
\newtheorem{lem}[thm]{Lemma}
\newtheorem{cor}[thm]{Corollary}
\newtheorem{claim}[thm]{Claim}
\newtheorem*{claim*}{Claim}

\theoremstyle{definition}
\newtheorem{defn}[thm]{Definition}
\newtheorem{ex}[thm]{Example}

\newtheorem{fact}[thm]{Fact}

\newtheorem{setting}[thm]{Setting}

\theoremstyle{remark}
\newtheorem*{pf}{{\sl Proof}}

\numberwithin{equation}{thm}
\def\Hom{\mathrm{Hom}}

\def\Ext{\mathrm{Ext}}

\def\a{\mathfrak a}

\def\e{\mathrm{e}}
\def\m{\mathfrak m}

\def\q{\mathfrak q}

\def\H{\mathrm{H}}

\newcommand{\rmg}{\mathrm{g}}

\newcommand{\calD}{\mathcal{D}}

\newcommand{\calM}{\mathcal{M}}

\newcommand{\fka}{\mathfrak{a}}
\newcommand{\fkb}{\mathfrak{b}}

\newcommand{\fkm}{\mathfrak{m}}

\newcommand{\fkp}{\mathfrak{p}}
\newcommand{\fkq}{\mathfrak{q}}

\newcommand{\fkM}{\mathfrak{M}}

\def\depth{\mathrm{depth}}
\def\Supp{\mathrm{Supp}}

\def\a_i{\underline {a_i}}

\def\Ass{\mathrm{Ass}}
\def\Assh{\mathrm{Assh}}
\def\Min{\mathrm{Min}}

\begin{document}
\setlength{\baselineskip}{10.9pt}

\title{ On  Hilbert coefficients  and sequentially Cohen-Macaulay rings }

\author[K. Ozeki]{Kazuho Ozeki}
\address{Department of Mathematical Sciences, Faculty of Science, Yamaguchi University, 1677-1 Yoshida, Yamaguchi 753-8512, Japan}
\email{ozeki@yamaguchi-u.ac.jp}

\author[H. L. Truong]{Hoang Le Truong}
\address{Institute of Mathematics, VAST, 18 Hoang Quoc Viet Road, 10307
Hanoi, Viet Nam}
\address{Thang Long Institute of Mathematics and Applied Sciences, Hanoi, Vietnam}
\email{hltruong@math.ac.vn\\
	truonghoangle@gmail.com}

\author[H. N.Yen]{Hoang Ngoc Yen}
\address{Institute of Mathematics, VAST, 18 Hoang Quoc Viet Road, 10307
Hanoi, Viet Nam}
\address{The Department of Mathematics, Thai Nguyen University of education.
20 Luong Ngoc Quyen Street, Thai Nguyen City, Thai Nguyen Province, Viet Nam.}
\email{hnyen91@gmail.com}

\thanks{
\endgraf
The first author was partially supported by Grant-in-Aid for Scientific Researches (C) in Japan (18K03241). The second author was partially supported by the Alexander von Humboldt Foundation and the Vietnam National Foundation for Science and Technology Development (NAFOSTED) under grant number 101.04-2019.309. The last author was partially supported by  Grant number  ICRTM02-2020.05,
awarded in the internal grant competition of International Center for Research and Postgraduate Training in Mathematics, Hanoi.
}

\date{}

\maketitle

\begin{abstract}
In this paper, we explore the relation between the index of reducibility and the Hilbert coefficients in local rings. Consequently, the main result of this study provides a characterization of a sequentially Cohen-Macaulay ring in terms of  its Hilbert coefficients for non-parameter ideals. As corollaries to the main theorem, we obtain characterizations of a Gorenstein/Cohen-Macaulay ring in terms of its Chern coefficients for parameter ideals.
\end{abstract}

\section{Introduction}

Throughout this paper, let $(R,\m)$ be a homomorphic image of a Cohen-Macaulay local ring with the infinite residue field $k$, $\dim R=d>0$, and $M$ a finitely generated $R$-module of dimension $s$. 
For an $\m$-primary ideal $I$ of $M$, it is well-known  that there are integers $\e_i(I;M)$, called the {\it Hilbert coefficients} of $M$ with respect to $I$ such that for $n \gg 0$
\begin{eqnarray*}
 \ell_R(M/{I^{n+1}}M)={\e}_0(I;M) \binom{n+s}{s}-{\e}_1(I;M) \binom{n+s-1}{s-1}+\cdots+(-1)^s {\e}_s(I;M).
\end{eqnarray*}
Here $\ell_R(N)$ denotes, for an $R$-module $N,$ the length of $N.$ In particular, the leading coefficient $\e_0(I)$ is said to be {\it the multiplicity} of $M$ with respect to $I$ and $\e_1(I)$, which Vasconselos (\cite{Vas08}) refers to as the {\it Chern number} of $M$ with respect to $I$. 
Now our motivation stems from the work of  \cite{Vas08}. Vasconcelos posed {\it the Vanishing Conjecture}:  $R$ is a Cohen-Macaulay local ring if and only if $\e_1(\q, R) = 0$ for some parameter ideal $\q$ of $R$.
It is shown that the relation between Cohen-Macaulayness and the Chern number of parameter ideals is quite surprising.  In \cite{Tru14}, motivated by some deep results of \cite{CGT13,TaT20 } and also by the fact that this is true for $R$ is unmixed as shown in \cite{GGHOPV10}, it was asked whether  the characterization of many classes of non-unmixed rings such as Buchsbaum rings, generalized Cohen-Macaulay rings, sequentially Cohen-Macaulay rings in terms of the Hilbert coefficients of {\it non-parameter ideals} of $R$. The aim of our paper is to continue this research direction. 
Concretely,  we will give characterizations of a sequentially Cohen-Macaulay ring in terms of its Hilbert coefficients with respect to certain non-parameter ideals (Theorem \ref{main1}). Recall that the notion of a sequentially Cohen-Macaulay module was first introduced  by Stanley (\cite{Sta96}) for the graded case.  In the local case, a ring $R$ is said to be a {\it sequentially Cohen-Macaulay ring} if there exists a filtration of ideals
$\calD : \fka_0=(0)\subseteq \fka_1\subsetneq \fka_2\subsetneq \cdots 
\subsetneq \fka_{\ell}=R$ such that $\textrm{ht } \fka_i < \textrm{ht } \fka_{i+1}$ and $\fka_{i+1}/\fka_{i}$ are Cohen-Macaulay for all $i = 0, 1, \ldots , \ell - 1$, where $\textrm{ht } \fka$ is the height of an ideal $\fka$ (\cite{Sch99}). Then $R$ is a Cohen-Macaulay ring if and only if $R$ is an unmixed sequentially Cohen-Macaulay ring. 
Therefore, as an immediate consequence of our main result, we get again a result which is slightly stronger than the main results in \cite{Tru17} (Theorem \ref{T6.6}). Furthermore, Theorem \ref{main1} allows us to get several interesting properties of the fiber cone of socle ideals (Theorem \ref{T1.2}). Finally, 
from this main result, we obtain characterizations of a  Cohen-Macaulay ring in terms of its irreducible multiplicity with respect to certain parameter ideals (Theorem \ref{T1.3}).

This paper is divided into four sections. In the next section we recall the notions of the dimension filtration,   distinguished parameter ideals following \cite{CuC07,Tru19}, and prove some preliminary results on the index of reducibility. In  Section 3, we explore the relation between the index of reducibility and the Hilbert coefficients. The last section is devoted to prove the main result and its consequences.

\section{Preliminary}

Throughout this paper let  $R$  be a commutative Noetherian local  ring with  maximal ideal $\m$. Assume that the residue field $k=R/\fkm$ is infinite and $\dim R=d$. Let $M$ be a finitely generated $R$-module of dimension $s >0$.  We put $r_j(M) = \ell(0:_{\H_{\m}^{j}(M)} \m)$, for all $j\in \mathbb{Z}$, and $\Assh_RM=\{\fkp \in \Supp_RM \mid \dim R/\fkp=s\}.$ Then $\Assh_RM \subseteq \Min_RM \subseteq \Ass_RM.$
Let $\Lambda (M)=\{\dim_R L \mid L \ \text{is an}\ R\text{-submodule of}\ M, L \ne (0)\}.$ We then have $$\Lambda (M)=\{\dim R/\fkp \mid \fkp \in\Ass_RM\}.$$
We put  $\ell=\sharp \Lambda (M)$ and number the elements  $\{d_i\}_{1\leq i\leq \ell}$ of $\Lambda (M)$  
so that  $0 \le d_1<d_2<\cdots<d_{\ell}=s.$
We denote by $\Ass_{i}(M)=\{\fkp\in\Ass M\mid \dim R/\fkp\le d_i\}$. Then because the base ring $R$ is Noetherian, 
for each $1\leq i\leq \ell$ the $R$-module $M$ contains 
the largest $R$-submodule $D_i$ with $\Ass D_i=\Ass_i(M)$. 
Therefore, letting $D_0=(0)$, we have the filtration 
$$\calD : D_0=(0)\subsetneq D_1\subsetneq D_2\subsetneq \cdots 
\subsetneq D_{\ell}=M$$
of $R$-submodules of $M$, which we call the dimension filtration of $M$. 
The notion of the dimension filtration was first given 
by P. Schenzel (\cite{Sch99}). 
Our notion of the dimension filtration is a little different from that of \cite{CuC07, Sch99},
but throughout  this paper let us utilize the above definition. It is standard to check that $\{D_j\}_{0\leq j\leq i}$ 
(resp. $\{D_j/D_i\}_{i\leq j\leq \ell}$) 
is the dimension filtration of $D_i$ (resp. $M/D_i$) 
for every $1\leq i\leq \ell$. We put   $C_i=D_i/D_{i-1}$ for $1 \le i \le \ell$. 

We note two characterizations of the dimension filtration. Let $(0)=\bigcap_{\fkp \in \Ass_RM}M(\fkp)$ be a primary decomposition
of $(0)$ in $M$, where $M(\fkp)$ is an $R$-submodule of $M$ 
with $\Ass_RM/M(\fkp)=\{ \fkp \}$ for each $\fkp \in \Ass_RM$. Then the submodule
$D_{\ell-1} =\bigcap\limits_{\fkp\in\Assh(M)}M(\fkp)$
is called the unmixed component of $M$.

\begin{fact}[{\cite[Proposition 2.2, Corollary 2.3]{Sch99}}]\label{d1}
	For all $i=1,\ldots,\ell$, we have 
	\begin{enumerate}
		\item[$(1)$] $D_i=\bigcap\limits_{\fkp \in \Ass_RM,\ \dim R/\fkp \geq d_{i+1}}M(\fkp)$,
				\item[$(2)$] $\Ass_RC_i=\{ \fkp \in \Ass_RM \mid \dim R/\fkp =d_i\}$,
		\item[$(3)$] $\Ass_RM/D_i=\{\fkp \in \Ass_RM \mid \dim R/\fkp 
		\geq d_{i+1} \}$, 
		\item[$(4)$] $\dim D_i=d_i$.
	\end{enumerate}
\end{fact}

Now, let $\underline x=x_1,x_2, \ldots, x_s$ be a system of parameters of $M$  and $\q_j$ denote the ideal generated by $x_1,\ldots,x_j$ for all $j=1,\ldots,s$. Then $\underline x$ is said to be 
{\it distinguished}, if  
$  (x_j \mid d_{i} < j \le s) D_i=(0)$
for all $1\le i\le \ell$. A parameter ideal $\fkq$ of $M$ is called  
{\it distinguished}, if there exists a 
distinguished system $x_1,x_2, \ldots, x_s$ of parameters of $M$ such that $\fkq=(x_1,x_2, \ldots, x_s)$. Therefore, if $M$ is a Cohen-Macaulay $R$-module, every parameter ideal of $M$ is distinguished. 
Distinguished system of parameters always exist 
and if $x_1, x_2, \ldots, x_s$ is a distinguished system of parameters of $M$, so are $x_1^{n_1},x_2^{n_2}, \ldots, x_s^{n_s}$ for all integers $n_j \ge 1$ and $j=1,\ldots,s$.

\medskip

\begin{defn}[{cf. \cite[Definition 2.3]{Tru19}}]\rm
	A system $x_1,x_2,\ldots,x_t$ of elements of $R$ is called a {\it Goto sequence} on $M$, if   for all $0\le j\le t-1$ and $0\le i\le\ell$, we have
	\begin{enumerate}
		\item[$(1)$] $\Ass(C_i/\q_j C_i)\subseteq \Assh(C_i/\q_j C_i)\cup \{\fkm\}$,
		\item[$(2)$] $x_jD_i=0$ if $d_i<j\le d_{i+1}$,
		\item[$(3)$] $(0):_{M/\q_{j-1}M}x_j=\H^0_\m(M/\q_{j-1}M) \text{ and } x_j\not\in\fkp \text{ for all } \fkp\in\Ass(M/\q_{j-1}M)-\{\fkm\}$.
	\end{enumerate}
	
\end{defn}

 We now assume that $R$ is a homomorphic image of a Cohen-Macaulay local ring. Then
 the existence of Goto sequences is established in  (\cite[Corollary 2.8]{Tru19}).

	\medskip

\begin{fact}[{\cite[Lemma 2.4, Lemma 2.5]{Tru19}}]\label{property}
	Let $\underline x=x_1,x_2,\ldots,x_t$ form a Goto sequence on $M$. Then the following assertions hold true.
	\begin{enumerate}
		\item[$(1)$] $\underline x$ is a part of a system of parameters of $M$. 
		\item[$(2)$] If $t=s$ then $\underline x$ is a distinguished system of parameters of $M$.
		\item[$(3)$] Let $N$ denote the unmixed component of $\overline M=M/\q_{t-2}M$. If $\overline M/N$ is Cohen-Macaulay, so is also $M/D_{\ell-1}$. 
	\end{enumerate}
	
\end{fact}

The notion of a sequentially Cohen-Macaulay module was first introduced  by Stanley (\cite{Sta96}) for the graded case, and in \cite{Sch99} for the local case.

\begin{defn}[\cite{Sch99, Sta96}]
	We say that $M$ is a {\it sequentially Cohen-Macaulay $R$-module}, 
	if $D_i/D_{i-1}$ is a Cohen-Macaulay $R$-module for all 
	$1\leq i \leq \ell$. 
\end{defn}

Towards the end of this section, we develop the basic theory of  the index of reducibility. The use of the index of reducibility is an important part of our approach to the Hilbert coefficients; we show that there are very important connections between the index of reducibility and the Hilbert coefficients in the next section. Recall that an $R$-submodule $N$ of $M$ is {\it irreducible}, if $N$ is not written as the intersection of two larger $R$-submodules of $M$. Every $R$-submodule $N$ of $M$ can be expressed as an irredundant intersection of irreducible $R$-submodules of $M$ and the number of irreducible $R$-submodules appearing in such an expression depends only on $N$ and not on the expression (\cite{Noe21}). Let us call, for each $\m$-primary ideal $I$ of $M$, the number $\mathrm{ir}_M(I)$ of irreducible $R$-submodules of $M$ that appear in an irredundant irreducible decomposition of $I M$ is called the index of reducibility of  $I$ on $M$. Notice that, we have 
 $\mathrm{ir}_M(I) =\ell_R([I M :_M \fkm]/I M).$
 For  a parameter ideal $\q$ of $M$, several properties of $\mathrm{ir}_M(\q)$ had been found and played essential roles in the earlier stage of development of the theory of Gorenstein rings and/or Cohen-Macaulay rings. Recently, the index of reducibility of parameter ideals has been used to deduce a lot of information on the structure of some classes of modules, such as regular local rings by W. Gr\"{o}bner (\cite{Gro51}); Gorenstein rings by Northcott, Rees (\cite{Noe21,Nor57,Tru14,Tru19}); Cohen-Macaulay modules by D.G. Northcott, N.T. Cuong, P.H. Quy (\cite{CQT15,Tru14,Tru17,Tru19}); Buchsbaum modules by S. Goto, N. Suzuki and H. Sakurai (\cite{GoS03,GoS84}); generalized Cohen-Macaulay modules by N.T. Cuong, P.H. Quy and the second author (\cite{CuQ11, CuT08}), and see also \cite{Tru13, Tru19,TaT20} for other modules. 
 The following theorem is to give a characterization of a sequentially Cohen-Macaulay module in terms of its the index of reducibility of parameter ideals. 

\begin{thm}[{\cite[Theorem 1.1]{Tru13,Tru19}}]\label{3.800}
	Assume that $R$ is a homomorphic image of a Cohen-Macaulay
	local ring. Then the following statements are equivalent.
\begin{enumerate} 
		\item[$({\rm i})$] $M$ is sequentially Cohen-Macaulay.
		
		\item[$({\rm ii})$]  There exists an integer $n \gg 0$ such that for every distinguished parameter ideals $\q$ of $M$ contained in $\m^{n}$, we have
$\mathrm{ir}_M(\q)=\sum\limits_{j \in \Bbb Z}r_j(M).$
\end{enumerate} 
\end{thm}

In the following lemma, we will give some properties of distinguished parameters.
\begin{lem}\label{F2.5}
Let $M$ be a sequentially Cohen-Macaulay $R$-module. Assume that $\q=(\underline x)$ is  a distinguished  parameter ideal of $M$ such that
$\mathrm{ir}_M(\q)=\sum\limits_{j \in \Bbb Z}r_j(M)$. Let $\fkb$ be an ideal generated by a part of a system $\underline x$ of parameters. 
Then  for all $n\ge 0$, we have
$$(\q^n M+D_{\ell-1}+\fkb M):_M \fkm=\q^n M:_M\fkm+D_{\ell-1}+\fkb M
 \text{ and } \q^{n+1}M:\m= \q^n(\q M:\m)+(0):_M\m.$$

\end{lem}
\begin{proof}
  Let $N=D_{\ell-1}$, $\calM= M/\fkb M+N$, $t=\dim \calM$ and  $\operatorname{gr}_{\q}(\calM)=\bigoplus\limits_{n\ge 0}\frak q^n\calM/\frak q^{n+1}\calM$. 
 Since $\calM$ is Cohen-Macaulay, we have a natural isomorphism of graded modules
$${\rm gr}_\q(\calM):=\bigoplus\limits_{n\ge 0}\frak q^n\calM/\frak q^{n+1}\calM\to \calM/\frak q \calM[X_{1},\ldots,X_t],$$ where $\{X_{i}\}_{i=1}^t$ are indeterminates.
After applying the functor $\Hom(k,\bullet)$, we get
$  {\mathfrak{q}^{n+1}\mathcal{M}:\mathfrak{m} } =  {\mathfrak{q}^{n}\mathcal{M}(\mathfrak{q}\mathcal{M}:\mathfrak{m}) }.$
 Since $\frak q$ is the parameter ideal of Cohen-Macaulay module $\calM$, we get
  $\frak q^{n+1}\calM:\frak m\subseteq \frak q^{n+1}\calM:\frak q=\frak q^{n}\calM.$
  It follows that 
$\frak
q^{n+1} \calM:\frak m=\frak q^{n}\calM(\frak q \calM:_\calM\frak m)$ and so we have
$$(\frak q^{n+1}M+N+\fkb M):\frak m = \frak q^n((\frak qM+N):\frak m)+N+\fkb M.$$
 By \cite[Proposition 2.8]{CGT13a}, we have $(\q M+N):_M\m=\q M:_M\m+N,$ and so
 we obtain
$$(\frak q^{n+1}M+N+\fkb M):\m \subseteq \q^n(\frak q M:\frak
m)+N+\fkb M\subseteq \frak q^{n+1}M:\frak m+N+\fkb M.$$ Hence we have
$(\frak q^{n+1}M+N+\fkb M):\m = \q^n(\frak q M:\frak
m)+N+\fkb M= \frak q^{n+1}M:\frak m+N+\fkb M$
for all $n\ge 0$.

Now, 
let $a\in\q^{n+1}M:\m$. Then $a\in \q^n (\q M:\m)+N$
 and we write $a=b+c$ for $b\in\q^n (\q M:\m)$ and $c\in N$. Then $\m c=\m (a-b)\in \q^{n+1}M\cap N$.   Since $\q$ is a parameter ideal of Cohen-Macaulay module $M/N$, we have $\q^{n+1}M\cap N=\q^{n+1}N$ and so we get $c\in \q^{n+1}N:\m$. 
Hence $\q^{n+1}M:\m\subseteq \q^n(\q M:\m)+\q^{n+1}N:\m$. Moreover, we obtain 
$\q^{n+1}M:\m
=\q^n(\q M:\m)+\q^{n+1} N:\m.
$Now if $\dim N=0$, we have $\q N=0$, because of the definition of distinguished parameter ideals. Then  $\q^{n+1}M:\m
=\q^n(\q M:\m)+(0):_M\m$.\par 

If $\dim N> 0$, since $M$ is sequentially Cohen-Macaulay and $\q$ is the distinguished parameter ideal of $M$,  $N$ is sequentially Cohen-Macaulay and $\q$ is a distinguished parameter ideal of $N$ such that
$\mathrm{ir}_N(\q)=\sum\limits_{j \in \Bbb Z}r_j(N)$ (\cite[the proof of Theorem 1.1]{Tru13}).
By the induction on $\ell$, we have
$\q^{n+1} N:\m=\q^n[\q N:\m]+(0):_N\m$. Therefore, we have $\q^{n+1}M:\m=\q^n(\q M:\m)+(0):_M\m,$
as required.
\end{proof}

\begin{prop}\label{P2.7}
Let $M$ be a sequentially Cohen-Macaulay module of dimension $s$. Assume that $\q$ is  a distinguished  parameter ideal of $M$ such that
$\mathrm{ir}_M(\q)=\sum\limits_{j \in \Bbb Z}r_j(M).$ Then for all $n\ge 1$, we have
$$\mathrm{ir}_M(\q^{n+1}) = \sum\limits_{i=1}^sr_i(M)\binom{n+i-1}{i-1}+r_0(M).$$

\end{prop}
\begin{proof}
Let $N=D_{\ell-1}$ and $L=M/N$.  Since $\q$ is a parameter ideal of Cohen-Macaulay module $L$,  the  sequence 
$$0\to N/\frak q^{n+1} N\to M/\frak q^{n+1}M\to L/\frak q^{n+1} L \to0$$
are exact.
It follows from  $[\frak
q^{n+1}M+N]:\frak m= \frak q^{n+1}M:\frak m+N$ by the Lemma \ref{F2.5}  and
applying $\Hom_R(k,*)$ to the above sequence that we obtain the  exact sequence
$$0\to \Hom_R(k,N/\frak q^{n+1} N)\to \Hom_R(k,M/\frak q^{n+1}M)\to \Hom_R(k,L/\frak q^{n+1} L)\to0.$$
Therefore, we have
$\ell_R([\frak q^{n+1}M:\frak m]/\frak q^{n+1}M)=\ell_R([\frak q^{n+1} N:\frak m]/\frak q^{n+1} N)+\ell_R([\frak q^{n+1} L:\frak m]/\frak q^{n+1} L).$
Since $L$ is Cohen-Macaulay  by \cite[Theorem 1.1]{Tru14} or \cite[Theorem 5.2]{CQT15}, we get 
$\ell_R([\frak q^{n+1} L:\frak m]/\frak q^{n+1} L)  =
r_s(L)\binom{n+s-1}{s-1}.$
Since $M$ is sequentially Cohen-Macaulay, so is $N$ and we have $r_s(L)= r_{s}(M)$ and $r_{i}(N) = r_{i}(M)$ for all $i \le s-1.$  Moreover, $\q$ is a distinguished parameter ideal of $N$ such that
$\mathrm{ir}_N(\q)=\sum\limits_{j \in \Bbb Z}r_j(N)$ (\cite[the proof of Theorem 1.1]{Tru13}).
By the induction on $\ell$, we have
$$\begin{aligned}
\mathrm{ir}_M(\q^{n+1})
= \mathrm{ir}_L(\q^{n+1})+\mathrm{ir}_N(\q^{n+1})
 &=r_s(L)\binom{n+s-1}{s-1}+\sum\limits_{i=1}^{d_{\ell-1}}r_i(M)\binom{n+i-1}{i-1}+r_0(N)\\
&= \sum\limits_{i=1}^sr_i(M)\binom{n+i-1}{i-1}+r_0(M),
\end{aligned}
$$ for all $n\ge 1$. Thus, the proof is complete.

\end{proof}

\section{The Hilbert coefficients of socle ideals}

In this section, we explore the relation between the index of reducibility and Hilbert coefficients. 
We maintain the following settings in Section 3 and Section 4.

\begin{setting}\label{2.6}{\rm
		Assume that $R$ is a homomorphic image of a Cohen-Macaulay local ring. Let $\calD = \{\fka_i\}_{0 \le i \le \ell}$ be the dimension filtration of $R$ with $\dim \fka_i=d_i$. We put $C_i = \fka_{i}/\fka_{i-1}$, $S = R/\fka_{\ell -1}$.		 Let 
		$\q$ be a distinguished parameter ideal of $R$, which is  generated by 
		 a distinguished system $\underline x=x_1,x_2,\ldots,x_d$  of parameters of $R$. Then we set $I=\q:\m$, $\fkb=(x_{d_{\ell-1}},\ldots,x_d)$,
		 and $\q_j=(x_1,\ldots,x_j)$ for all $j=1,\ldots,d$.   
}
\end{setting}

 For each integer $n\geq 1$, we denote by ${\underline x}^n$ the sequence $x^n_1, x^n_2,\ldots,x^n_d$. Let $K^{\bullet}(x^n)$ be the
Koszul complex of $R$ generated by the sequence ${\underline x}^n$ and let
$H^{\bullet}({\underline x}^n;R) = H^{\bullet}(\Hom_R(K^{\bullet}({\underline x}^n),R))$
be the Koszul cohomology module of $R$. Then for every $p\in\Bbb Z$, the family $\{H^p({\underline x}^n;R)\}_{n\ge 1}$
naturally forms an inductive system of $R$, whose limit
$H^p_\fkq(R)=\lim\limits_{n\to\infty} H^p({\underline x}^n;R)$
is isomorphic to the local cohomology module
$H^p_\fkm(R)=\lim\limits_{n\to\infty} \Ext_R^p(R/\fkm^n,R).$
For each $n\geq 1$ and $p \in\Bbb Z$, let $\phi^{p,n}_{{\underline x},R}:H^p({\underline x}^n;R)\to H^p_\fkm(R)$ denote the canonical
homomorphism into the limit.


\begin{defn}[{\cite[Lemma 3.12]{GoS03}}]\label{sur}
	{\rm There exists an integer $n_0$ 
		such that for all systems of parameters ${\underline x}=x_1,\ldots,x_d$  for $R$ contained in $\fkm^{n_0}$ and for all $p\in \Bbb Z$, the canonical homomorphisms
		$$\phi^{p,1}_{{\underline x},R}:H^p({\underline x},R)\to H^p_\fkm(R)$$
		into the inductive limit are surjective on the socles.
		The least integer $n_0$ with this property is called a \textit{g-invariant} of $R$ and denote by $\rmg(R)$.}
\end{defn}


The next lemmas establish  certain properties of  \textit{g-invariant} and distinguished parameter ideals $\q \subseteq \m^{\rmg(R)}$.

\begin{lem}[{\cite[Lemma 2.12]{Tru19}}]\label{2.10}
	Assume that $S$ is Cohen-Macaulay. Then  for all parameter ideals $\q\subseteq\m^{\rmg(R)}$ of $R$, we have 
	$$\mathrm{ir}_R(\q)=\mathrm{ir}_{\fka_{\ell-1}}(\q)+\mathrm{ir}_{S}(\q).$$
	
\end{lem}

\begin{lem}\label{3.444}
		Assume that $S$ is Cohen-Macaulay and $\q=(x_1,\ldots,x_d)$ is a distinguished parameter ideal of $R$ such that $\q \subseteq \m^{\rmg(R)}$. Let $\fkb=(x_{d_{\ell-1}},\ldots,x_d)$.  
	Then  we have $\rmg(R/\fkb)\le \rmg(R).$
\end{lem}
\begin{proof}
Let $\overline R=R/\fkb$.  For an ideal $J$ of $R$, we denote $\overline J=(J+\fkb)/\fkb$. 
Let $\overline Q$ be a parameter ideal of $\overline R$ such that $\overline Q\subseteq \overline \m^{\rmg(R)}$. Then we have
 $\fkb\subseteq Q\subseteq \fkm^{\rmg(R)}$, where $Q$ is a preimage of the ideal $\overline Q$  in $R$.  By the definition of  g-invariant,  the canonical maps
$\phi^{p}_{Q,R}:H^p(Q,R)\to H^p_\fkm(R)$
	are surjective on the socles.  Then we look at the exact sequence
	\begin{equation}\label{eqn1}
	\xymatrix{0\ar[r]& \fka_{\ell-1}\ar[r]^\iota& R \ar[r]^\epsilon & S \ar[r]&0}
	\end{equation}
	of $R$-modules, where $\iota$ (resp. $\epsilon$) denotes the canonical embedding (resp. the canonical epimorphism).
	Since $\dim \fka_{\ell-1} < d$ and  $S$ is Cohen-Macaulay, we get the following commutative diagram 
	$$\xymatrix{
		0\ar[r]& \fka_{\ell-1}/Q \fka_{\ell-1}\ar[r]^{\overline \iota}& R/Q  \ar[d]^{\phi_R}\ar[r]^{\overline \epsilon} & S/Q S \ar[r] \ar@{^{(}->}[d]^{\phi_{S}}&0\\
		& & \H^d_\m(R) \ar[r]^{=} & \H^d_\m(S) &}$$
with  exact first row. 
	Let $x\in (0):_{S/QS}\m$. Then, since $\phi_R^d$ is surjective on the socles, we get  an element $y\in (0):_{R/Q R}\m$ such that $\phi_{S}^d(x)=\phi_R^d(y)$. Thus $\overline{\epsilon}(y)=x$, because the canonical map $\phi_{S}^d$ is injective, whence
	$[\fka_{\ell-1}+Q]:_R \fkm=\fka_{\ell-1} + [Q :_M\fkm].$ Therefore,
we
have 	$$\ell((0):_{R/Q}\fkm)=\ell((0):_{\fka_{\ell-1}/Q \fka_{\ell-1}}\fkm)+\ell((0):_{S/Q S}\fkm).$$
The exact sequence \ref{eqn1} induces, for all $p< d$, the commutative diagram
$$\xymatrix{
		H^p(Q,\fka_{\ell-1})\ar@{^{}->}[d]^{\phi_{Q,\fka_{\ell-1}}^p}\ar[r]&H^p(Q,R)\ar[d]^{\phi^{p}_{Q,R}}\\
		\H^p_\m(\fka_{\ell-1}) \ar[r]^{=} & \H^p_\m(R).}$$
Since $\phi^{p}_{Q,R}$ are surjective on the socles for all $p$,  $\phi^{p}_{Q,\fka_{\ell-1}}$ are surjective on the socles for all $p<d$. Since $x_{d_{\ell-1}},\ldots,x_d$ is a regular sequence of $S$, it follows from the exact sequence \ref{eqn1} and $\fka_{\ell-1}\cap \fkb=0$ that the sequence
	\begin{equation}\label{eqn3}
\xymatrix{0\ar[r]& \fka_{\ell-1}\ar[r]&R/\fka\ar[r]&S/\fkb S\ar[r]&0}
\end{equation}
is exact.
Therefore, we have the commutative diagram
$$\xymatrix{
		0\ar[r]& \fka_{\ell-1}/Q \fka_{\ell-1}\ar[r]\ar[d]^{\phi_{Q,\fka_{\ell-1}}^{d_{\ell-1}}}& R/Q  \ar[d]^{\phi_{Q,R/\fkb}^{d_{\ell-1}}}\ar[r] & S/Q S \ar[r] \ar@{^{(}->}[d]^{\phi_{Q,S/\fkb S}^{d_{\ell-1}}}&0\\
		0\ar[r]&\H^{d_{\ell-1}}_\m(\fka_{\ell-1})\ar[r] & \H^{d_{\ell-1}}_\m(R/\fkb ) \ar[r] & \H^{d_{\ell-1}}_\m(S/\fkb S) \ar[r]&0 }$$
with  exact rows.  By applying the functor $\Hom(k;\bullet)$ to this diagram, we obtain a commutative diagram
	\begin{equation}\label{eqn2}
\xymatrix{
		0\ar[r]& (0):_{\fka_{\ell-1}/Q \fka_{\ell-1}}\fkm\ar[r]^{\overline \iota}\ar[d]^{\overline{\phi_{Q,\fka_{\ell-1}}^{d_{\ell-1}}}}& (0):_{R/Q }\fkm \ar[d]^{\overline{\phi_{Q,S/\fkb S}^{d_{\ell-1}}}}\ar[r]^{\overline \epsilon} & (0):_{S/Q S}\fkm \ar[r] \ar@{^{(}->}[d]^{\overline{\phi_{Q,S/\fkb S}^{d_{\ell-1}}}}&0\\
		0\ar[r]&(0):_{\H^{d_{\ell-1}}_\m(\fka_{\ell-1})}\fkm\ar[r] & (0):_{\H^{d_{\ell-1}}_\m(R/\fkb R)}\fkm \ar[r] & (0):_{\H^{d_{\ell-1}}_\m(S/\fkb S)}\fkm\ar[r] &0 }
			\end{equation}
		of complexes of $R$-modules. Since $\ell((0):_{R/Q}\fkm)=\ell((0):_{\fka_{\ell-1}/Q \fka_{\ell-1}}\fkm)+\ell((0):_{S/Q S}\fkm)$, the above row of the commutative diagram \ref{eqn2} is exact. Since $S$ is Cohen-Macaulay,  $\overline{\phi_{Q,S/\fkb S}^{d_{\ell-1}}}$ is surjective. Thus, the lower row of the  commutative diagram \ref{eqn2}  is exact. Since $d_{\ell-1}<d$,  $\overline{\phi_{Q,\fka_{\ell-1}}^{d_{\ell-1}}}$ is surjective and so  is $\overline{\phi_{Q,R/\fkb}^{d_{\ell-1}}}$.
The exact sequence \ref{eqn3} induces the commutative diagram
$$\xymatrix{
		H^p(Q,\fka_{\ell-1})\ar@{^{}->}[d]^{\phi_{Q,\fka_{\ell-1}}^p}\ar[r]&H^p(Q,R/\fkb)\ar[d]^{\phi^{p}_{Q,R}}\\
		\H^p_\m(\fka_{\ell-1}) \ar[r]^{=} & \H^p_\m(R/\fkb),}$$
		for all $p< d_{\ell-1}$.
Since $\phi_{Q,\fka_{\ell-1}}^{p}$ are surjective on the socles for all $p<d$,  $\overline{\phi_{Q,R/\fkb}^{p}}$ are surjective for all $p<d_{\ell-1}$. Hence $\phi_{Q,R/\fkb}^{p}$ are surjective on the socles for all $p$. Thus, we have $\rmg(R/\fkb)\le \rmg(R)$, as required.
		
\end{proof}

\begin{lem}\label{tcr}
For all distinguished parameter ideal $\q \subseteq \m^{\rmg(R)}$ and $j=0,\ldots d$,  we have
	$$  r_d(R)\le r_{d-j}(R/\q_j).$$ 
\end{lem}
\begin{proof}
It follows from the exact sequence 
	$$\xymatrix{0\ar[r]&(0):_Rx_1 \ar[r] &R\ar[r]^{\ .x_1}&R\ar[r]&R/\q_1\ar[r]&0}$$
	that we have two commutative diagrams
\begin{center}
	$\xymatrix{R/\q \ar[r]^f\ar[d]^{g_1}&R/\q +(0):_Rx_1\ar[d]^g\\
		\H^{d-1}_\fkm(R/\q_1)\ar[r]^{f_1^\prime}&\H^{d}_\fkm(R)}$
	and $\xymatrix{R/\q \ar[r]^f\ar[d]^{g_0}&R/\q +0:_Rx_1\ar[d]^g\\
		\H^{d}_\fkm(R)\ar[r]^{\cong}&\H^{d}_\fkm(R).}$
	
\end{center}
	 Thus, it is easy to check that the diagram
		$$\xymatrix{R/\q \ar[d]^{g_1}\ar[rd]^{g_0}&\\
			\H^{d-1}_\fkm(R/\q_1)\ar[r]^{f_1}&\H^{d}_\fkm(R)}$$
			commutes, where $f_1=f_1^\prime$. Similarly, we have the commutative diagram
		 		$$\xymatrix{R/\q \ar[d]^{g_2}\ar[rd]^{g_1}&\\
		 			\H^{d-2}_\fkm(R/\q_2)\ar[r]^{f_2^\prime}&\H^{d-1}_\fkm(R/\q_1)}$$
		and so that the diagram
		 		$$\xymatrix{R/\q \ar[d]^{g_2}\ar[rd]^{g_0}&\\
		 			\H^{d-2}_\fkm(R/\q_2)\ar[r]^{f_2}&\H^{d}_\fkm(R)}$$
		 		 	 commutes, where $f_2=f_1\circ f_2^\prime$.	By induction on $j$, we have
		 		 	the commutative diagram
		 		 	$$\xymatrix{R/\q \ar[d]^{g_j}\ar[rd]^{g_0}&\\
		 		 		\H^{d-j}_\fkm(R/\q_j)\ar[r]^{f_j}&\H^{d}_\fkm(R),}$$ 
		 		 	where $f_j=f_{j-1}\circ f_j^\prime$.		 	
	After applying the functor $\Hom(k,\bullet)$, we obtain the commutative diagram 
		 		$$\xymatrix{(\q :\m)/\q \ar[d]^{g_i^*}\ar[rd]^{g_0^*}&\\
		 			0:_{\H^{d-j}_\fkm(R/\q_j)}\m\ar[r]^{f_i^*}&0:_{\H^{d}_\fkm(R)}\m.}$$
	 Since the map $g_0^*$ is surjective, so is the map $$f_j^*:\Hom(k,\H^{d-j}_\fkm(R/\q_j))\to\Hom(k,\H^{d}_\fkm(R)).$$ Therefore, we have $r_d(R)\le r_{d-j}(R/\q_j)$
	for all $j=0,\ldots d$. This completes the proof.
	
\end{proof}

Now, we  begin our study of the Hilbert coefficients of distinguished parameter ideals of $R$.

\begin{prop}\label{coe}
Assume that $R$ is sequentially Cohen-Macaulay
and $\q$ is a distinguished parameter ideal of $R$ such that $\mathrm{ir}_R(\q)=\sum\limits_{i \in \Bbb Z}r_i(R).$
Then 
we have
$$ e_j(I) - e_j(\mathfrak{q}) = (-1)^{j-1}r_{d-j+1}(R),$$
for all $j=0,\dots,d$, if $e_0(\m;R)> 1$ or $\q\subseteq \m^2$ and $d\ge 2$.
\end{prop}

\begin{proof}
	By Theorem 3.2 and Corollary 3.5 in \cite{CGT13a}, we have $I^2=\fkq I$ and so $I^{n+1}=\fkq^n I$ for all $n\ge 1$. Since $\mathfrak{q}^{n+1} \subseteq I^{n+1}$,   we have the  exact sequence
$$ 0 \to I^{n+1}/ \mathfrak{q}^{n+1} \to R/ \mathfrak{q}^{n+1} \to R/I^{n+1} \to 0.$$
	Thus, by Lemma \ref{F2.5}, we have
	\begin{eqnarray*}
\ell \left( R/ \mathfrak{q}^{n+1}\right) -  \ell \left( R/ I^{n+1}\right) =   \ell \left( \mathfrak{q}^{n}I/ \mathfrak{q}^{n+1}\right) 
 =  \ell \left( \frac{\mathfrak{q}^{n}(\mathfrak{q}:\mathfrak{m}) }{\mathfrak{q}^{n+1}}\right) = \ell \left( \frac{\mathfrak{q}^{n+1}:\mathfrak{m}}{\mathfrak{q}^{n+1}}\right) - \ell ((0): \m).
\end{eqnarray*}
By comparing the coefficients of the polynomials and Proposition \ref{P2.7}, we obtain 
$$ e_j(\q:\m) - e_j(\mathfrak{q}) = (-1)^{j-1}r_{d-j+1}(R),$$
for all $j=0,\dots,d$.

\end{proof}

\begin{cor}\label{4.300}
Suppose that $R$ is a sequentially Cohen-Macaulay ring. Then  there exists an integer $n \gg 0$ such that for every distinguished parameter ideals $\q\subseteq \m^{n}$ and $j=0,\dots,d$, we have
$$ e_j(I) - e_j(\mathfrak{q}) = (-1)^{j-1}r_{d-j+1}(R).$$
\end{cor}
\begin{proof}
This is  immediate from Proposition \ref{coe} and Theorem \ref{3.800}.
\end{proof}

\begin{lem}\label{4.5}
Assume that $S$ is Cohen-Macaulay and  $\q$  is a distinguished parameter ideal of $R$ such that
$[\q+\fka_{\ell-1}]:\fkm=\q:\m+\fka_{\ell-1}.$
 Then 
 we have
$$ e_{j}(I;R/\fkb)-e_{j}(\fkq;R/\fkb)=
 \begin{cases}(-1)^{d-d_{\ell-1}}((e_{d-d_{\ell-1}+j}(I;R)-e_{d-d_{\ell-1}+j}(\q;R)))+r_d(R)&\text{if $j=1$,}
\\
(-1)^{d-d_{\ell-1}} (e_{d-d_{\ell-1}+j}(I;R)-e_{d-d_{\ell-1}+j}(\q;R))&\text{if } j\ge 2.\\
 \end{cases}$$  
  
\end{lem}
\begin{proof}

Let $t=d-d_{\ell-1}$. Since $[\q+\fka_{\ell-1}]:\fkm=\q:\m+\fka_{\ell-1}$,  we have $IS=\q S:\m S$.

\begin{claim}\label{calim 1}
$(I^n+\fkb)\cap \fka_{\ell-1}=I^n\cap \fka_{\ell-1}$, for all $n\ge 1$.
\end{claim}
\begin{proof}
First, we show that $[I^n+(x_d)]\cap \fka_{\ell-1} =I^n\cap \fka_{\ell-1}$. Indeed, since $\fkq\subseteq \m^2$ and $S$ is Cohen-Macaulay, we have $(IS)^2 = (\q S)(IS)$ by   \cite[Theorem 3.7]{CHV98}, so that $\operatorname{gr}_{IS}(S)$ is a Cohen-Macaulay ring. Therefore, we have $I^nS:x_d = I^{n -1}S$, for all $n \in \Bbb Z$. Thus,  we have $(I^n+\fka_{\ell-1}):x_d=I^{n-1}+\fka_{\ell-1}$.

Let $a\in [I^n+(x_d)]\cap \fka_{\ell-1}$. Write $a=b+x_dc$ for some $b\in I^n$, $c\in R$. Then $c\in (I^n+\fka_{\ell-1}):x_d=I^{n-1}+\fka_{\ell-1}$. Thus since $\q$ is a distinguish parameter ideal of $R$, we have $x_d\fka_{\ell-1}=0$ and so
$x_dc\in  I^n$. Therefore $a\in I^n\cap \fka_{\ell-1}$. Hence $[I^n+(x_d)]\cap \fka_{\ell-1} =I^n\cap \fka_{\ell-1}$. 

Now, we shall demonstrate our result by induction on $t=d-d_{\ell-1}$. In the case $t = 1$ there
is nothing to prove. So assume inductively that $t > 1$ and  the desired result has been established when $t-1$.  Let $\overline R = R/(x_d)$. For an ideal $J$ of $R$, we denote $\overline J=(J+(x_d))/(x_d)$. 
 Then  
$\overline \fka_{\ell-1}$ is the unmixed component of $\overline R$. Thus,
$\overline R/\overline \fka_{\ell-1}=S/(x_d)S$ is a  Cohen-Macaulay ring  and $\overline \q$ is a distinguished parameter ideal of $R$ such that
$[\overline\q+\overline\fka_{\ell-1}]:\overline\fkm=\overline\q:\overline\m+\overline\fka_{\ell-1}.$
By the hypothesis of induction, 
we have $(\overline I^n+\overline\fkb)\cap \overline\fka_{\ell-1}=\overline I^n\cap \overline\fka_{\ell-1}$ for all $n\ge 1$. Thus, $[I^n+\fkb]\cap [\fka_{\ell-1}+(x_d)]=[I^n+(x_d)]\cap [\fka_{\ell-1}+(x_d)]$ for all $n\ge 1$.
Therefore, we have
$$\begin{aligned}[]
[I^n+\fkb] \cap \fka_{\ell-1} &=  [I^n+\fkb]\cap  [\fka_{\ell-1}+(x_d)]\cap \fka_{\ell-1}\\
&=[I^n+(x_d)]\cap [\fka_{\ell-1}+(x_d)] \cap \fka_{\ell-1}\\
&=[I^n+(x_d)] \cap \fka_{\ell-1}= I^n\cap \fka_{\ell-1}
\end{aligned}
$$
for all $n\ge 1$.
\end{proof}

 It follows from the above claim  that the sequences
 $$\xymatrix{0\ar[r]& \fka_{\ell-1}/(I^n\cap \fka_{\ell-1})\ar[r]& R/(\fkb+I^n)\ar[r]& S/(\fkb+I^n)S\ar[r]& 0}$$
$$\text{ and }\xymatrix{0\ar[r]& \fka_{\ell-1}/(I^n\cap \fka_{\ell-1})\ar[r]& R/I^n\ar[r]& S/I^nS\ar[r]& 0}$$
 are exact for all $n\ge 0$.
Therefore we have
$$\ell(R/I^n)-\ell(S/I^nS)=\ell(R/(\fkb+I^n))-\ell(S/(\fkb+I^n)S)$$
Since $S$ is Cohen-Macaulay, $S$ and $S/\fkb S$ are sequentially Cohen-Macaulay.
By Proposition \ref{coe}, we have
$$\begin{aligned}
\ell(S/I^nS)&=e_0(I;S)\binom{n+d}{d}-r_d(S)\binom{n+d-1}{d-1}\\
\text{ and \quad}
\ell(S/(\fkb+I^n)S)&=e_0(I;S)\binom{n+d_{\ell-1}}{d_{\ell-1}}-r_d(S)\binom{n+d_{\ell-1}-1}{d_{\ell}-1},
\end{aligned}$$
for all $n\ge 0$. Consequently,  on comparing the coefficients of the polynomials in the above equality, we have 
$$ e_{j}(I;R/\fkb)=
 \begin{cases}(-1)^{t}e_{t+1}(I;R)+r_d(S) &\text{if $j=1$,}
\\
(-1)^{t}e_{t+j}(I;R)&\text{if } j\ge 2.\\
 \end{cases}$$  
Similarly, we have $$ e_{j}(\q;R/\fkb)=(-1)^{t}e_{t+j}(I;R),$$  
for all $1\le j\le d_{\ell-1}$. It follows that
$$ e_{j}(I;R/\fkb)-e_{j}(\fkq;R/\fkb)=
 \begin{cases}(-1)^{t}((e_{t+j}(I;R)-e_{t+j}(\q;R)))+r_d(R)&\text{if $j=1$,}
\\
(-1)^{t} (e_{t+j}(I;R)-e_{t+j}(\q;R))&\text{if } j\ge 2.\\
 \end{cases}$$

\end{proof}

\begin{cor}\label{cor4.61}
Assume that $S$ is Cohen-Macaulay. 
 Then  we have 
$$ e_{j}(I;R/\fkb)-e_{j}(\fkq;R/\fkb)=
 \begin{cases}(-1)^{d-d_{\ell-1}}((e_{d-d_{\ell-1}+j}(I;R)-e_{d-d_{\ell-1}+j}(\q;R)))+r_d(R)&\text{if $j=1$,}
\\
(-1)^{d-d_{\ell-1}} (e_{d-d_{\ell-1}+j}(I;R)-e_{d-d_{\ell-1}+j}(\q;R))&\text{if } j\ge 2,\\
 \end{cases}$$  
   for all distinguished parameter ideals $\q\subseteq\m^{\rmg(R)}$ of $R$.

\end{cor}
\begin{proof}
This is immediate from Proposition \ref{4.5} and Lemma \ref{2.10}.
\end{proof}



\begin{prop}\label{P4.60}
Suppose that $\q$ is a distinguished parameter ideal of $R$ such that $\q \subseteq \m^{\rmg(R)}$ and
	$$ e_1(I)-e_1(\fkq)\leq r_{d}(R).$$
 Then  $S$ is Cohen-Macaulay.
\end{prop}
\begin{proof}
	In the case  $e_0(\m;R)=1$, we have  $e_0(\m;S)=1$, because $\dim \fka_{\ell-1}<\dim R$. And so the result in this case follows from $S$ is unmixed and Theorem 40.6 in \cite{Nag62}. Thus we suppose henceforth in this proof  that  $e_0(\m;R)>1$.

By Corollary 2.8. in \cite{Tru19}, 
there exists a Goto sequence $x_1, x_2, \ldots, x_{d}$ on $R$. Then by Fact \ref{property} 2),  $x_1, x_2, \ldots, x_d$ is a distinguished system of parameter of $R$ such that $\q=(x_1,x_2,\ldots,x_d)$ 
	satisfies the following conditions
	\begin{enumerate}
		\item[$(1)$]   $\Ass(C_i/\q_j C_i)\subseteq \Assh(C_i/\q_j C_i)\cup \{\fkm\}$, for all $j=1,\ldots,d-2$,
		\item[$(2)$]  $x_j$ are a superficial element of $R/\q_{j-1}$ with respect to $I$ and $\q$, for all $1\le j\le d-2$.
	\end{enumerate}
		Let  $B=R/\q_{d-1}$, $A=R/\q_{d-2}$ and let $N$ denote the unmixed component of $A$. Then $A/N$ is unmixed and $\H^0_{\m}(A/N) = 0$.
	Since $e(\m;R)>1$, by Proposition 2.3 in \cite{GoS03}, we get that $\m I^n = \m\q^n$ for all $n\ge 1$. Therefore $I^n\subseteq \q^n:\m$ for all $n\ge 1$. Put $G=\bigoplus\limits_{n\ge 0}\q^n/\q^{n+1}$ and $\fkM=\m/\q\oplus\bigoplus\limits_{n\ge 1}\q^n/\q^{n+1}$. 
	Then we have $((0):_G\fkM)_n=[\q^n\cap (\q^{n+1}:\m)]/\q^{n+1}$ for all $n\ge 0$ and so that
	 $((0):_G\fkM)_n=(\q^{n+1}:\m)/\q^{n+1}$ for all large $n$.
	Because $x_1$ is a superficial element of $R$ with respect to $\q$, we have $\q^{n+1}:x_1=\q^n$ for all $n\ge 0$. 
	Set $\overline R = R/(x_1)$. For an  ideal $J$ of $R$, we denote $\overline J=(J+(x_1)/(x_1)$.
	It follows from the exact sequences
	$$\xymatrix{0\ar[r]& \frac{I^{n}\cap (\q^{n+1}:x_1)}{\q^{n}}\ar[r]& \frac{I^{n}}{\q^{n}}\ar[r]& \frac{I^{n+1}}{\q^{n+1}}\ar[r]& \frac{I^{n+1}}{x_1I^n+\q^{n+1}}\ar[r]& 0}$$
	$$\text{and} \quad \xymatrix{0 \ar[r]&\frac{I^{n+1}\cap (x_1)+\q^{n+1}}{x_1I^n+\q^{n+1}}\ar[r]&\frac{I^{n+1}}{x_1I^n+\q^{n+1}}\ar[r]& \frac{\overline {I^{n+1}}}{\overline{\q^{n+1}}}\ar[r]& 0},$$
	for 
	all $n\ge 0$ that
	$\ell({I^{n+1}}/{\q^{n+1}})-\ell({I^{n}}/{\q^{n}})=\ell({I^{n+1}}/{x_1I^n+\q^{n+1}})\ge \ell({\overline{I^{n+1}}}/{\overline{\q^{n+1}}}).$
	Therefore, we have
	$ e_1(I;R) - e_1(\q;R) \geq e_1(\overline I;\overline R) - e_1(\overline\q;\overline R).$
	Notice that $\overline I=\overline \q:\overline \m$, by induction,  we have
	$$ r_d(R)\ge e_1(I;R)-e_1(\q;R)\ge e_1(IB;B)-e_1(\q B;B).$$
	On the other hand, since $\dim B=1$, $B$ is sequentially Cohen-Macaulay.
	By Proposition 2.3 in \cite{GoS03} $I$  is contained in the integral closure of $\q$. Thus by Proposition \ref{P2.7} and Theorem \ref{3.800},  for large enough $n$,  we have
	$$\begin{aligned}
	e_1(IB;B)-e_1(\q B;B)&= \ell(B/\q^{n+1}B)-\ell(B/I^{n+1}B)
	=\ell(\frac{((x_d)B:\m B)^{n+1}}{(x_d^{n+1})B}) \\
	&\ge\ell(\frac{(x_d^{n+1})B:\m B}{(x_d^{n+1})B}) 
	=\mathrm{ir}_B(\q^{n+1})=r_1(B)+r_0(B)\ge r_1(B).
	\end{aligned}$$
	Since $\q\subset \m^{\rmg(R)}$, by Lemma \ref{tcr}, we have
	$r_d(B)=r_1(R)\le r_2(A).$	
	 Moreover, we also obtain the commutative diagram
	
				 		$$\xymatrix{(0):_{\H^{2}_\fkm(A)}\m\ar[rd]^{f^*_{d-2}}&\\
			 			(0):_{\H^{1}_\fkm(B)}\m\ar[u]^{f^{\prime *}_{d-2}}\ar[r]^{\cong}&(0):_{\H^{d}_\fkm(R)}\m.}$$
	Since $f_{d-2}^*$ is surjective, $(0):_{\H^{1}_\fkm(B)}\m$ is a direct summand of $(0):_{\H^{2}_\fkm(A)}\m$. Therefore  the map $f^{\prime *}_{d-2}:(0):_{\H^{1}_\fkm(B)}\m\to (0):_{\H^{2}_\fkm(A)}\m$ is injective.

	On the other hand, it follows from the exact sequence
	$$0\to N\to A\to A/N\to 0$$
	and $x_{d-1}$ is a regular of $A/N$ that we have  the exact sequence
	$$0\to N/(x_{d-1})\to A/(x_{d-1})\to A/(x_{d-1})+N\to 0.$$
	Therefore $\H^{1}_\fkm(B) =\H^{1}_\fkm(A/(x_{d-1})+N)$ and $\H^{2}_\fkm(A)=\H^{2}_\fkm(A/N)$. Thus the canonical map $\alpha:(0):_{\H^{1}_\fkm(A/(x_{d-1})+N)}\m\to (0):_{\H^{2}_\fkm(A/N)}\m$ is injective.
	 The  exact sequence
	$$\xymatrix{0\ar[r]& A/N \ar[r]^{.x_{d-1}}& A/N\ar[r]& A/(x_{d-1})+N\ar[r]& 0}$$
	induces 
$\H^0_\fkm( A/(x_{d-1})+N)=(0):_{\H^1_\fkm( A/N)}x_{d-1}$ and	the  exact sequence 
	$$\xymatrix{0\ar[r]& \H^1_\fkm( A/N)/x_{d-1}\H^1_\fkm( A/N)\ar[r]&\H^1_\fkm( A/(x_{d-1})+N)\ar[r]& (0):_{\H^2_\fkm( A/N)}x_{d-1}\ar[r]& 0}.$$
		After applying the functor $\Hom(k,\bullet)$, we obtain the commutative diagram 
			$$\xymatrix{0\ar[r]& (0):_{\H^1_\fkm( A/N)/x_{d-1}\H^1_\fkm( A/N)}\m\ar[r]& (0):_{\H^1_\fkm( A/(x_{d-1})+N)}\m \ar[r]^{\alpha}& 0:_{\H^2_\fkm( A/N)}\m }.$$
		Since map $\alpha$ is injective, we have $\H^1_\fkm( A/N)=x_{d-1}\H^1_\fkm( A/N)$, and so  $\H^1_\fkm( A/N)=0$. Hence $S$ is Cohen-Macaulay, because of the Lemma \ref{property} 3). The proof is complete.

\end{proof}

\section{Proof of main theorems and corollaries}
In this section, we begin to combine all of our lemmas from Section 2 and 3 to provide a proof of the main theorems in this study.

\begin{thm}\label{main1}
Assume that $R$ is a homomorphic image of a Cohen-Macaulay local ring with $\dim R = d \ge 1$. 
Then the following statements are equivalent.
\begin{itemize}
\item[${\rm (1)}$] $R$ is sequentially Cohen-Macaulay.

\item[${\rm (2)}$]  There exists a distinguished parameter ideal $\fkq\subseteq \fkm^{\rmg(R)}$  such that 
$$ (-1)^{d-j}(e_{d-j+1}(\q:\m)-e_{d-j+1}(\fkq))\le r_{j}(R),$$
for all $2\le j\in\Lambda(R)$.
\end{itemize}
\end{thm}

\begin{proof}
${\rm (1)}\Rightarrow {\rm (2)}$ follows from  Corollary \ref{4.300}.\\
${\rm (2)} \Rightarrow {\rm (1)}$	We use induction on the dimension $d$ of $R$. In the case  $\dim R = 2$, it is the result of Proposition \ref{P4.60}.  Suppose that $\dim R > 2$ and  our assertion holds true for $\dim R-1$.   Recall that $\fka_{\ell-1}$ is the unmixed component of $R$ and $\fkb=(x_{d_{\ell-1}+1},\ldots,x_d)$. Therefore, by definition we have $\fkb\subseteq \fkm^{\rmg(R)}$ and $\fkb \cap \fka_{\ell-1}=0$. Thus, $(\fka_i+\fkb)/\fkb=\fka_i$ for all $i=0,\ldots,\ell-1$, and so $\Lambda(R)-\{d\}\subseteq \Lambda(R/\fkb)$.
	On the other hand, since $d\in\Lambda(R)$, we obtain $e_1(I;R)-e_1(\q, R)\le r_d(R)$. By Proposition \ref{P4.60}, $S$ is Cohen-Macaulay. It follows from  the  exact sequence
	$$\xymatrix{0\ar[r]& \fka_{\ell-1}\ar[r]& R \ar[r]& S\ar[r]& 0}$$
	that the sequence
	\begin{equation}\label{seq1}
	\xymatrix{0\ar[r]& \fka_{\ell-1}\ar[r]& R/\fkb \ar[r]& S/\fkb S\ar[r]& 0}
	\end{equation}
	is exact. Hence $\Lambda(R/\fkb)=\Lambda(R)-\{d\}$.
	
	
	 By Corollary \ref{cor4.61}, we have $$ e_{j}(I;R/\fkb)-e_{j}(\fkq;R/\fkb)=
	\begin{cases}(-1)^{t}((e_{t+j}(I;R)-e_{t+j}(\q;R)))+r_d(R)&\text{if $j=1$,}
	\\
	(-1)^{t} (e_{t+j}(I;R)-e_{t+j}(\q;R))&\text{if } j\ge 2,\\
	\end{cases}$$  
	where $t=d-d_{\ell-1}$. However, it follows from $S$ is Cohen-Macaulay and the  exact sequence \ref{seq1}
	that the  sequence
	$$0\to \H^{d_{\ell-1}}_\m(\fka_{\ell-1})\to \H^{d_{\ell-1}}_\m(R/\fkb) \to \H^{d_{\ell-1}}_\m(S/\fkb S)\to 0$$
	is exact. Moreover, we have $\H^{d_{\ell-1}}_\m(R)\cong \H^{d_{\ell-1}}_\m(\fka_{\ell-1})$  
	and $\H^i_\m(R)\cong \H^i_\m(\fka_{\ell-1})\cong \H^i_\m(R/\fkb)$, for all $i< d_{\ell-1}$. Thus, we have $r_i(R)=r_i(R/\fkb)$ for all $i<d_{\ell-1}.$ 
Since $S$ is Cohen-Macaulay, we have the exact diagram
$$\xymatrix{
		0\ar[r]& \fka_{\ell-1}/\q \fka_{\ell-1}\ar[r]^{ \iota}& R/\q  \ar[d]^{\phi_{R/\fkb}}\ar[r]^{ \epsilon} & S/\q S \ar[r] \ar@{^{(}->}[d]^{\phi_{S/\fkb S}}&0\\
		0\ar[r]& \H^{d_{\ell -1}}_\m(\fka_{\ell-1}) \ar[r] &  \H^{d_{\ell -1}}_\m(R/\fkb) \ar[r] & \H^{d_{\ell -1}}_\m(S/\fkb) \ar[r]&0.}$$
By applying the functor $\Hom(k,\bullet)$, we obtain the commutative diagram 
$$\xymatrix{
		0\ar[r]& (0) :_{\fka_{\ell-1}/\q \fka_{\ell-1}} \m \ar[r]^{\overline \iota}& (0) :_{ R/\q} \m  \ar[d]^{\overline\phi_{R/\fkb}}\ar[r]^{\overline \epsilon} & (0) :_{S/\q S} \m  \ar@{^{(}->}[d]^{\overline\phi_{S/\fkb S}}\\
	 0\ar[r]& (0) :_{\H^{d_{\ell -1}}_\m(\fka_{\ell-1})} \m  \ar[r] & (0) :_{ \H^{d_{\ell -1}}_\m(R/\fkb) } \m \ar[r] & (0) :_{\H^{d_{\ell -1}}_\m(S/\fkb)} \m, &}$$
where the rows are exact.  Since $S$ is Cohen-Macaulay, by Lemma \ref{2.10}, we have 
$\mathrm{ir}_R(\q) = \mathrm{ir}_{\fka_{\ell -1}}(\q) + \mathrm{ir}_{S}(\q).$ Therefore, map $\overline \epsilon$ is surjective.  Since $\overline\phi_{S/\fkb S}$ is surjective, 
we obtain
$$	\begin{aligned}
		r_{d_{\ell-1}}(R/\fkb)&=r_{d_{\ell-1}}(\fka_{\ell-1})+r_{d_{\ell-1}}(S/\fkb S)
		=r_{d_{\ell-1}}(R)+r_d(S)\\
		&\ge(-1)^{s}(e_{t+1}(I;R)-e_{t+1}(\q;R))+r_d(S)
		\ge e_{1}(I;R/\fkb)-e_{1}(\fkq;R/\fkb).
		\end{aligned}$$
	Moreover, we have
	\begin{eqnarray*}
		r_{j}(R/\fkb)=r_{j}(R)&\ge&(-1)^{d-j}(e_{d-j+1}(I;R)-e_{d-j+1}(\q;R))\\
		&= &(-1)^{d_{\ell-1}-j}(e_{d_{\ell-1}-j+1}(I;R/\fkb)-e_{d_{\ell-1}-j}(\fkq;R/\fkb)),
	\end{eqnarray*}
	for all $2\le j\in\Lambda(R/\fkb)-\{ d_{\ell-1}\}$.
Hence, $r_{d_{\ell-1}-j+1}(R/\fkb)\ge (-1)^{j+1}(e_{j}(I;R/\fkb)-e_{j}(\fkq;R/\fkb))$, for all  $2\le j\in\Lambda(R/\fkb)$.

Now since $S$ is Cohen-Macaulay, by Lemma \ref{3.444}, we have $\rmg(R/\fkb)\le \rmg(R)$. Then $\fkq/\fkb\subseteq \fkm^{\rmg(R/\fkb)}(R/\fkb)$. 	By the induction hypothesis, $R/\fkb$ is sequentially Cohen-Macaulay. Hence  $R$ is sequentially Cohen-Macaulay, as required.

\end{proof}


The first consequence of Theorem \ref{main1} is to give a result slightly stronger than Theorem 1.1 in \cite{Tru17}. 
\begin{thm}\label{T6.6}
Assume that $R$ is unmixed and $\dim R\ge2$. Then the following statements are equivalent.
\begin{enumerate}[{(1)}] \rm
    \item {\it $R$ is Cohen-Macaulay.}

    \item {\it For some parameter ideal $\frak q\subseteq \m^{\rmg(R)}$, we
    have}
$ \e_1(\q:\m)-\e_1(\q)\le r_d(R).$
\end{enumerate}

\end{thm}
\begin{proof}
It follows immediately from Theorem \ref{main1} and the fact that if $R$ is unmixed then every parameter ideal of $R$ is distinguished.
\end{proof}

The following consequence of Theorem \ref{main1} provides a characterization of Gorenstein rings.

\begin{cor}
Assume that $R$ is unmixed and $\dim R\ge2$. Then the following statements are equivalent.
\begin{enumerate}[$\rm (1)$] \rm
    \item {\it $R$ is Gorenstein.}

    \item {\it For some parameter ideal $\frak q\subseteq \m^{\rmg(R)}$, we
    have}
$ \e_1(\q:\m)-\e_1(\q)=1.$
\end{enumerate}

\end{cor}

\begin{proof}

${\rm (1)}\Rightarrow {\rm (2)}$ follows from  Theorem 3.2 in \cite{Tru17}.\\
${\rm (2)} \Rightarrow {\rm (1).}$	
Let $\q$  be a parameter ideal such that $\e_1(I) -\e_1(q) =1 $  and $\frak q\subseteq \m^{\rmg(R)}$, then we have
$\e_1(I) - \e_1(\q) \le r_d(R).$
 By Theorem \ref{main1}, $R$ is Cohen-Macaulay. Therefore, we have
$r_d(R) = \e_1(I) - \e_1(\q) = 1.$
 Hence, $R$ is Gorenstein, as required.

\end{proof}

\begin{cor}\label{P4.3}
	Assume that  $\dim R \geq 2$. Then for all  distinguished parameter ideals $\q \subseteq \fkm^{\rmg(R)}$, we have
	$  r_d(R)\le e_1(\q:\fkm)-e_1(\fkq).$
\end{cor}

\begin{proof}
The result follows from Theorem \ref{main1}.
\end{proof}

Recall that the {\it fiber cone} of $I$ is the graded ring
$F(I)=\bigoplus_{n\ge 0} I^n/\m I^n.$
It is well-known in this setting that the Hilbert function $H_F(n)$ giving the dimension of $I^n/\m I^n$ as a vector space over $k$ is defined for  sufficiently large $n$ by a polynomial $h_F(X) \in \Bbb Q[X]$, the Hilbert polynomial of $F(I)$ \cite[Corollary, page 95]{Mat86}. A simple application of Nakayama’s lemma, (\cite[Theorem 2.2]{Mat86}), shows that the cardinality of a minimal set of generators of $I^n$ and denoted by $\mu(I^n)$, is equal to $\ell(I^n/\m I^n)$, the value of the Hilbert function $H_F(n)$ of $F(I)$. 
Then the integers $f_i(I;R)$ exist such that
$$h_F(n+1)=\sum\limits_{i=0}^{d-1}(-1)^if_i(I;R)\binom{n+d-1-i}{d-1-i}.$$
From the notations given above, the second main application of Theorem \ref{main1} is stated as follows.

\begin{thm}\label{T1.2}
 Assume that $R$ is a non-regular unmixed  local ring and $\dim R\ge 2$. Then the following statements are equivalent.
	\begin{itemize}
		\item[${\rm (1)}$] $R$ is Cohen-Macaulay.
		\item[${\rm (2)}$] For some parameter ideal $\fkq \subseteq \fkm^{\rmg(R)}$, we
		have
		$f_0(\frak q:\fkm;R)  = r_d(R)+1.$
				
	\end{itemize}
	
\end{thm}
\begin{proof}
	(1) $\Rightarrow$ (2). Let $I=\q:\m$. Since $R$ is a non-regular unmixed  local ring, $e_0(\fkm;R)>1$ (\cite[Theorem 40.6]{Nag62}). Thus, by Proposition 2.3 in \cite{GoS03}, we obtain $\fkm I^n = \fkm\fkq^n$ for all $n$.	Then, since $R$ is Cohen-Macaulay, by Lemma \ref{F2.5}, we have
	$$\begin{aligned}
	\ell(I^{n+1}/\m I^{n+1})& = \ell(I^{n+1}/\q^{n+1})+\ell(\q^{n+1}/\m\q^{n+1})
	&=\ell([\q^{n+1}:\m]/\q^{n+1})+\binom{n+d-1}{d-1}.
	\end{aligned}$$
	Therefore, by Proposition \ref{P2.7}, we have $f_0(I;R)=r_d(R)+1$.
	
	
	(2) $\Rightarrow$ (1). 
	Put $I=\q:\m$. Since $e_0(\fkm;R)>1$, by Proposition 2.3 in \cite{GoS03}, we obtain $\fkm I^n = \fkm\fkq^n$ for all $n$. Therefore, we obtain
	$$\ell(R/\fkq^{n+1})-\ell(R/I^{n+1})=\ell(I^{n+1}/\fkq^{n+1})=\ell(I^{n+1}/\m I^{n+1})-\ell(\q^{n+1}/\m\q^{n+1}).$$ 
	However, this means that 
	$$\begin{aligned}
	e_1(I;R)-e_1(\fkq;R)= f_0(I;R) -f_0(\fkq;R)
	\le  f_0(I;R) -1=r_d(R).
	\end{aligned}$$
Since $R$ is unmixed, by Theorem \ref{main1}, $R$ is Cohen-Macaulay, as required.	
\end{proof}

Notice that N. T. Cuong et al. (\cite{CQT15}) showed that there exists  integers $\{g_i(\q;R)\}_{i=0}^{d-1}$ such that   $$\mathrm{ir}_R(\q^n)=\ell_R([\q^{n+1} :_R \fkm]/\q^{n+1} )=\sum\limits_{i=0}^{d-1}(-1)^ig_i(\q;R)\binom{n+d-1-i}{d-1-i},$$ for sufficiently large $n$. 
The leading coefficient $g_0(\q;R)$ is  called {\it the irreducible multiplicity} of $\q$.
With the above notations, the third main application of Theorem \ref{main1} is stated as follows.

\begin{thm}\label{T1.3} Assume that $R$ is unmixed and $\dim R\ge 2$. Then the following statements are equivalent.
	\begin{itemize}
		\item[${\rm (1)}$] $R$ is Cohen-Macaulay.
		\item[${\rm (2)}$] For some parameter ideal $\fkq \subseteq \fkm^{\rmg(R)}$, we
		have
		$g_0(\frak q;R)  = r_d(R).$
				
	\end{itemize}
	
\end{thm}
\begin{proof}
	(1) $\Rightarrow$ (2). This is  immediate from Proposition \ref{P2.7} and the definition of sequentially Cohen-Macaulay.
	
	
	(2) $\Rightarrow$ (1). Since, if $R$ is unmixed and $e_0(\fkm;R)=1$, then $R$ is Cohen-Macaulay (\cite[Theorem 40.6]{Nag62}), and it follows that it is sufficient for us to prove this result under the addition hypothesis that $e_0(\fkm;R)>1$.
	
	Since $e_0(\fkm;R)>1$, by Proposition 2.3 in \cite{GoS03}, we obtain $\fkm I^n = \fkm\fkq^n$, for all $n$. Therefore, $I^n\subseteq \fkq^n:\fkm$ for all $n$. Thus, we obtain
	$$\ell(R/\fkq^{n+1})-\ell(R/I^{n+1})=\ell(I^{n+1}/\fkq^{n+1})\le \ell((\fkq^{n+1}:\fkm/\fkq^{n+1}).$$ 
	However, this means that $e_1(I;R)-e_1(\fkq;R)\le g_0(\fkq;R)$. Since $ g_0(\fkq;R)=r_d(R)$, we have $e_1(I;R)-e_1(\fkq;R)\le r_d(R).$ Since $R$ is unmixed, by Theorem \ref{main1}, $R$ is Cohen-Macaulay, as required.
	
\end{proof}

Let us note the following example to illustrate our arguments.

\begin{ex}
 Let $d \geqslant 3$ be an integer and let $U = k[[X_1,X_2, \ldots ,X_d, Y ]]$ be the
formal power series ring over a field $k$. We look at the local ring
$R = U/[(X_1,X_2, \ldots ,X_d) \cap (Y )]$.
Then $R$ is a reduced ring with $\dim R = d$. Moreover, $R$ is sequentially Cohen-Macaulay.
We put $A = U/(Y )$ and $D =
U/(X_1,X_2,\ldots ,X_d)$. 
Let $\fkq$ be a parameter ideal in $R$. Then, since $D$ is a DVR and $A$ is a
regular local ring with $\dim A = d$, thanks to the exact sequence $0\to D \to R \to A \to 0$,
 we get that $\depth R = 1$ and  the sequence
$$\xymatrix{0\ar[r]&D/\fkq^{n+1}D\ar[r]&R/\fkq^{n+1}R\ar[r]&A/\fkq^{n+1}A\ar[r]&0}$$
is exact. By applying the functor
$\mathrm{Hom}_R(R/\frak m, \bullet)$ we obtain the following exact sequence
$$\xymatrix{0\ar[r]&[\fkq^{n+1}:_D\fkm]/\fkq^{n+1}\ar[r]&[\fkq^{n+1}:_R\fkm]/\fkq^{n+1}\ar[r]&[\fkq^{n+1}:_A\fkm]/\fkq^{n+1}\ar[r]&0}.$$
Therefore, we have
$$\begin{aligned}
\ell_R([\fkq^{n+1}:_R\fkm]/\fkq^{n+1})&=\ell_R([\fkq^{n+1}:_A\fkm]/\fkq^{n+1})+\ell_R([\fkq^{n+1}:_D\fkm]/\fkq^{n+1})
&=\binom{d-1+n-1}{d-1}+ 1,
\end{aligned}$$
for all integers $n \geqslant 0$, whence $f_0(\fkq;R)=1=r_d(R)$ for every parameter ideal $\fkq$ in $\fkm$.
Since $A$ is regular local ring and $d\ge 3$, we have 
$$e_1(\q:\m)-e_1(\q)=r_d(R).$$ 
Notice that $R$ is not a Cohen-Macaulay ring, since $H^1_\fkm(R) = H^1_\fkm(D)$ is not a finitely generated $R$-module,
where $\fkm$ denotes the maximal ideal in $R$.

\end{ex}

\end{document}